\documentclass[12pt]{amsart}
\usepackage{latexsym,amssymb,amsmath}
\usepackage{amsmath,amsfonts}
\usepackage{epsfig}
\usepackage{graphicx}
\usepackage{verbatim}
\newtheorem{theorem}{Theorem}[section]
\newtheorem{lemma}[theorem]{Lemma}
\newtheorem{prop}[theorem]{Proposition}

\allowdisplaybreaks

\newcommand{\Hmm}[1]{\leavevmode{\marginpar{\tiny%
$\hbox to 0mm{\hspace*{-0.5mm}$\leftarrow$\hss}%
\vcenter{\vrule depth 0.1mm height 0.1mm width \the\marginparwidth}%
\hbox to 0mm{\hss$\rightarrow$\hspace*{-0.5mm}}$\\\relax\raggedright #1}}}
\newcommand{\nc}{\newcommand}
\nc{\les}{\lesssim}
\nc{\ges}{\gtrsim}
\nc{\nit}{\noindent}
\nc{\nn}{\nonumber}
\nc{\D}{\partial}
\nc{\diff}[2]{\frac{d #1}{d #2}}
\nc{\diffn}[3]{\frac{d^{#3} #1}{d {#2}^{#3}}}
\nc{\pdiff}[2]{\frac{\partial #1}{\partial #2}}
\nc{\pdiffn}[3]{\frac{\partial^{#3} #1}{\partial{#2}^{#3}}}
\nc{\abs}[1] {\lvert #1 \rvert}
\nc{\cAc}{{\cal A}_c}
\nc{\cE}{{\cal E}}
\nc{\cF}{{\mathcal F}}
\nc{\cP}{{\cal P}}
\nc{\cV}{{\cal V}}
\nc{\cQ}{{\cal Q}}
\nc{\cGin}{{\cal G}_{\rm in}}
\nc{\cGout}{{\cal G}_{\rm out}}
\nc{\cO}{{\cal O}}
\nc{\Lav}{{\cal L}_{\rm av}}
\nc{\cL}{{\cal L}}
\nc{\cB}{{\cal B}}
\nc{\cZ}{{\cal Z}}
\nc{\cR}{{\cal R}}
\nc{\cT}{{\cal T}}
\nc{\cY}{{\cal Y}}
\nc{\cX}{{\cal X}}
\nc{\cXT}{{{\cal X}(T)}}
\nc{\cBT}{{{\cal B}(T)}}
\nc{\vD}{{\vec \mathcal{D}}}
\nc{\efield}{\mathcal{E}}
\nc{\vE}{{\vec \efield}}
\nc{\vB}{{\vec \mathcal{B}}}
\nc{\vH}{{\vec \mathcal{H}}}
\nc{\ty}{{\tilde y}}
\nc{\tu}{{\tilde u}}
\nc{\tV}{{\tilde V}}
\nc{\Pc}{{\bf P_c}}
\nc{\bx}{{\bf x}}
\nc{\bX}{{\bf X}}
\nc{\bXYZ}{{\bf XYZ}}
\nc{\bY}{{\bf Y}}
\nc{\bF}{{\bf F}}
\nc{\bS}{{\bf S}}
\nc{\dV}{{\delta V}}
\nc{\dE}{{\delta E}}
\nc{\TT}{{\Theta}}
\nc{\dPsi}{{\delta\Psi}}
\nc{\order}{{\cal O}}
\nc{\Rout}{R_{\rm out}}
\nc{\eplus}{e_+}
\nc{\eminus}{e_-}
\nc{\epm}{e_\pm}
\nc{\eps}{\varepsilon}
\nc{\vnabla}{{\vec\nabla}}
\nc{\G}{\Gamma}
\nc{\w}{\omega}
\nc{\mh}{h}
\nc{\mg}{g}
\nc{\vphi}{\varphi}
\nc{\tlambda}{\tilde\lambda}
\nc{\be}{\begin{equation}}
\nc{\ee}{\end{equation}}
\nc{\ba}{\begin{eqnarray}}
\nc{\ea}{\end{eqnarray}}

\nc{\g}{\gamma}
\nc{\ol}{\overline}

\def\R{\mathbb R}

\nc{\T}{\mathbb T}
\nc{\Z}{\mathbb Z}
\nc{\N}{\mathbb N}
\nc{\pt}{\partial_t}
\nc{\la}{\langle}
\nc{\ra}{\rangle}
\nc{\infint}{\int_{-\infty}^{\infty}}
\nc{\halfwidth}{6.5cm}
\nc{\uu}{\" u}
\nc{\oo}{\" o}
\nc{\nlayers}{L} \nc{\nsectors}{M}
\nc{\indicator}{\mathbf{1}}
\nc{\Rhole}{R_{\rm hole}}
\nc{\Rring}{R_{\rm ring}}
\nc{\neff}{n_{\rm eff}}
\nc{\Frem}{F_{\rm rem}}
\nc{\DD}{\Delta}
\nc{\cD}{\mathcal D}
\nc{\lnorm}{\left\|}
\nc{\rnorm}{\right\|}
\nc{\rnormp}{\right\|_{\ell^{p,\eps}}}
\nc{\rar}{\rightarrow}
\nc{\sgn}{{\rm sign}}
\nc{\non}{\nonumber}
\nc{\wh}{\widehat}
\allowdisplaybreaks
\date{\today}

\begin{document}

\title[Fractional Schr\"odinger equation]{Smoothing  for the fractional Schr{\" o}dinger equation on the torus and the real line}

\author[ Erdo\u{g}an, G{\uu}rel, Tzirakis]{M. B. Erdo\u{g}an, T. B. G{\uu}rel, and N. Tzirakis}
\thanks{The first author is partially supported by NSF grant  DMS-1501041. The second author is supported by a grant from the Fulbright Foundation, and thanks the University of Illinois for its hospitality. The third  author's work was supported by a grant from the Simons Foundation (\#355523 Nikolaos Tzirakis)}
 
\address{Department of Mathematics \\
University of Illinois \\
Urbana, IL 61801, U.S.A.}
\email{berdogan@illinois.edu}

\address{Department of Mathematics \\
University of Illinois \\
Urbana, IL 61801, U.S.A.}
\curraddr{Department of Mathematics, Bo\u gazi\c ci University, Bebek 34342, Istanbul, Turkey}
\email{bgurel@boun.edu.tr}

\address{Department of Mathematics \\
University of Illinois \\
Urbana, IL 61801, U.S.A.}
\email{tzirakis@illinois.edu}

\begin{abstract}

In this paper we study the cubic fractional nonlinear Schr\"odinger equation (NLS) on the torus and on the real line. Combining the normal form and the restricted norm methods we prove that the nonlinear part of the solution is smoother than the initial data. Our method applies to both focusing and defocusing nonlinearities. In the case of full dispersion (NLS) and on the torus, the gain is a full derivative, while on the real line we get a derivative smoothing with an $\epsilon$ loss. Our result lowers the regularity requirement of a recent theorem of Kappeler et al., \cite{kst} on the periodic defocusing cubic NLS, and extends it to the focusing case and to the real line. We also obtain estimates on the higher order Sobolev norms of the global smooth solutions in the defocusing case.

\end{abstract}

\maketitle

\section{Introduction}

In this paper we study the  fractional cubic Schr\"odinger equation:
\begin{equation}\label{sch}
\left\{
\begin{array}{l}
iu_{t}+(-\Delta)^{\alpha} u =\pm |u|^2u, \,\,\,\,  x \in {\mathbb{K}}, \,\,\,\,  t\in \mathbb{R} ,\\
u(x,0)=u_0(x)\in H^{s}(\mathbb{K}), \\
\end{array}
\right.
\end{equation}
where $\alpha \in (1/2,1]$, and $\mathbb K=\T$ or $\R$. Note that the case $\alpha=1$ is known as the cubic NLS equation, and it is a very well studied model.   
The equation \eqref{sch} is called defocusing when the sign in front of the nonlinearity is a minus and focusing when the sign is a plus.
Fractional Schr\"odinger equations posed on the real line and on the torus have appeared in many recent articles, see \cite{kay}, \cite{det}, \cite{kwon} and the references therein. 
For example it has been observed that in the context of quantum mechanics,  the path integral formalism over Brownian trajectories
leads to NLS type equations. Whereas,  the path integral over
L\'evy trajectories leads to the fractional Schr\"odinger equation, see Laskin, \cite{laskin}. A rigorous derivation of the equation can be found in \cite{kay} starting from a family of models describing charge transport in bio polymers like the DNA. The starting point is a discrete
nonlinear Schr\"odinger equation with general lattice interactions. Equation \eqref{sch} with $\alpha\in(\frac12,1)$ appears as the continuum limit of the long--range interactions between quantum particles on the lattice. 

Bourgain in \cite{bou} obtained  the local and global wellposedness of the periodic NLS equation which corresponds to the case $\alpha=1$. More precisely he proved the existence and uniqueness of local-in-time strong $L^2(\T)$ solutions. Since it is known that smooth solutions of  NLS satisfy mass conservation 
$$M(u)(t)=\int_{\mathbb{T}}|u(t,x)|^2 dx=M(u)(0),$$
Bourgain's result leads to the existence of global-in-time strong $L^2(\T)$ solutions in the focusing and defocusing case. The $L^2(\T)$ theorem of Bourgain is sharp since as it was shown in \cite{bgt}, the solution operator is not uniformly continuous on $H^s(\Bbb T)$ for $s<0$. For the existence and uniqueness of global $L^2(\R)$ solutions the reader can consult \cite{tsu}. 

For $\alpha \in (\frac12,1)$, the  local and global wellposedness theory  on the torus was developed in \cite{det} utilizing the restricted norm method (also known as the $X^{s,b}$ method) of Bourgain, \cite{bou}. For the cubic equation in one dimension the critical regularity index is $s_c=\frac{1-\alpha}{2}$, and the theory developed in \cite{det} proves local wellposedness for any $s>s_c$. The endpoint was obtained in \cite{kwon}. In \cite{kwon},  the authors also established the local wellposedness theory on the real line for any $s \geq s_c$. Moreover they proved that the equation  is illposed for any $s<s_c$ both on the torus and the real line. A recent paper, \cite{hs}, discusses the fractional Schr\"odinger equation in higher dimensions and with more general power type nonlinearities. In what follows we rely on  the $X^{s,b}$ wellposedness theory that was established in \cite{det} and \cite{kwon}.

In this paper we improve the work in \cite{det} where the smoothing properties of the fractional cubic equation were studied. The case $\alpha=1$  was considered earlier in  \cite{talbot}. Both papers utilized the restricted norm method of Bourgain to prove  that the nonlinear part of the solution is smoother than the linear part. The gain in regularity was up to   $\min(\alpha-\frac12, 2s+\alpha-1).$ 
The result on the torus was improved to a full derivative gain by Kappaler, Schaad and Topalov,  in \cite{kst},  for the defocusing cubic NLS whenever the initial data is smoother than $H^2(\T)$. Their paper  contains many other interesting and nontrivial results, and is based on complete integrability methods, that are applicable only in the defocusing case.  The results  in \cite{talbot} and \cite{det} on the other hand are true for focusing or defocusing equations and are independent of the integrability structure of the system.

In addition to the conservation of mass, smooth solutions of \eqref{sch} satisfy energy conservation  
\be\label{conserve}
E(u)(t)=\int_{\mathbb{K}}\big||\nabla|^{\alpha} u(t,x)\big|^2 dx \pm \frac{1}{2}\int_{\mathbb{K}}\big|u(t,x)\big|^4 dx =E(u)(0).
\ee
Note that local theory at  $H^\alpha$ level along with the conservation of mass and energy imply the existence of global-in-time energy solutions. Since the equation is mass and energy sub-critical, \cite{tao}, one can obtain  global solutions also in the focusing case. This easily follows from the Gagliardo-Nirenberg inequality
$$\|u\|_{L^4}^4\lesssim \||\nabla|^{\alpha}u\|_{L^2}^{\frac{1}{\alpha}}\|u\|_{L^2}^{4-\frac{1}{\alpha}}$$
which controls the potential energy via the kinetic energy $\||\nabla|^{\alpha}u\|_{L^2}$. One can then control the Sobolev norm of the solution for all times even in the focusing case  since $\frac{1}{\alpha}<2$. In \cite{det}, the authors used the smoothing estimates in the $X^{s,b}$ norms to prove a global wellposedness  result for initial data with infinite energy.

In this paper we combine the $X^{s,b}$ theory with the theory of normal forms as was developed in \cite{bit} and \cite{et} for the periodic KdV equation. We improve the nonlinear smoothing estimates in \cite{talbot} and  \cite{det} significantly, and obtain a gain of regularity of order
$$\min(2\alpha-1, 2s+\alpha-1),$$ 
see Theorem~\ref{thm:1} below. We should mention  that for any $s>1/2$ and for the cubic NLS equation, the  gain  in \cite{talbot} was half a derivative while now we obtain a full derivative gain on the torus and almost a full derivative on the real line. This decreases the regularity requirement ($s\geq 2$) of a smoothing theorem in \cite{kst}, and extends it to the focusing case and to the real line for $\alpha\in (\frac12,1]$.

Below we summarize our results. On the torus our smoothing estimate  reads as follows.
   
\begin{theorem}\label{thm:1} Consider the equation \eqref{sch} on $\T$.  Fix $\frac12<\alpha\leq 1.$ For any $s> \frac34-\frac\alpha2 $, and $a\leq \min(2\alpha-1,2s+\alpha-1) $ (the inequality has to be strict if the minimum is  $2s+\alpha-1$), we have  
$$
\|u(t)-e^{it(-\Delta)^\alpha-iPt}u_0\|_{C^0_tH^{s+a}_x}\les \|u_0\|_{H^s}^3 +\|u_0\|_{H^s}^5
$$
for $t<T$, where $T$ is the local existence time, and $P=\frac{1}{\pi}\|u_0\|_2^2$.
\end{theorem}
In the case $s>\frac12$, the proof of Theorem~\ref{thm:1} uses a normal form transform and a priori bounds in $H^s$ spaces (see Proposition~\ref{prop:main} and Proposition~\ref{prop:mainvar}),  completely  avoiding the  $X^{s,b}$ spaces. Therefore, it    can be upgraded to the energy estimate below. Estimates of this form are useful especially when $s \geq \alpha$. In such cases energy conservation is used to bound the $H^{\alpha}$ norm and thus obtain a global-in-time energy smoothing estimate that can be used to obtain polynomial-in-time bounds for higher order Sobolev norms (see Theorem \ref{thm:n4} below). 

\begin{theorem} \label{thm:2} Consider the equation \eqref{sch} on $\T$.  Fix $\frac12<\alpha\leq 1.$ For any $s> \frac12 $, and $a\leq 2\alpha-1 $, we have   $u(t)-e^{it(-\Delta)^\alpha-iPt}u_0\in C^0_tH^{s+a}_x$ and 
\begin{multline*}
\|u(t)-e^{it(-\Delta)^\alpha-iPt}u_0\|_{ H^{s+a} } \les  
\|u_0\|_{H^s}(1+\|u_0\|_{H^{\frac12+}}^2)  +\|u(t)\|_{H^s} \|u(t)\|_{H^{\frac12+}}^2\\ +\int_0^t \|u(t^\prime)\|_{H^s}\big(\|u(t^\prime)\|_{H^{\frac12+}}^2+\|u(t^\prime)\|_{H^{\frac12+}}^4\big) dt^\prime
\end{multline*}
for all $t$ in the maximal interval of existence.

 In particular, for $s\geq \alpha$, and $a\leq 2\alpha-1 $, we have  
$$
\|u(t)-e^{it(-\Delta)^\alpha-iPt}u_0\|_{H^{s+a}} \les  
\|u_0\|_{H^s}   +\|u(t)\|_{H^s}  \\ +\int_0^t \|u(t^\prime)\|_{H^s} dt^\prime,
$$
for all $t$. 
\end{theorem}
 Our next theorem establishes a similar result on the real line.  
\begin{theorem} \label{thm:n3} Consider the equation \eqref{sch} on $\R$.  Fix $\frac12<\alpha\leq 1.$ For any $s> \frac12 $, and $a < 2\alpha-1 $, we have   $u(t)-e^{it(-\Delta)^\alpha }u_0\in C^0_tH^{s+a}_x$ and 
\begin{multline*}
\|u(t)-e^{it(-\Delta)^\alpha }u_0\|_{ H^{s+a} } \les  
\|u_0\|_{H^s}(1+\|u_0\|_{H^{\frac12+}}^2)  +\|u(t)\|_{H^s} \|u(t)\|_{H^{\frac12+}}^2\\ +\int_0^t \|u(t^\prime)\|_{H^s}\big(\|u(t^\prime)\|_{H^{\frac12+}}^2+\|u(t^\prime)\|_{H^{\frac12+}}^4\big) dt^\prime
\end{multline*}
for all $t$ in the maximal interval of existence.

 In particular, for $s\geq \alpha$, and $a < 2\alpha-1 $, we have  
$$
\|u(t)-e^{it(-\Delta)^\alpha }u_0\|_{H^{s+a}} \les  
\|u_0\|_{H^s}   +\|u(t)\|_{H^s}  \\ +\int_0^t \|u(t^\prime)\|_{H^s} dt^\prime,
$$
for all $t$. 
\end{theorem}
We should mention that on the real line the normal form method can be, in principle, harder to apply.   An example of this phenomenon is the KdV equation where one can combine the normal form with $X^{s,b}$ estimates to obtain nonlinear smoothing on the torus, which fails for the corresponding problem on the real line, see e.g. \cite{etbook}. In the case of the fractional NLS we manage to overcome this problem by carefully estimating the small frequency interaction of the nonlinear wave functions (resonances) and obtain a similar result on both the real line and the torus.  
Concerning  Theorem~\ref{thm:n3}, we remark that one can probably lower the regularity requirement for the nonlinear smoothing on the real line, and obtain a smoothing estimate for $s\leq \frac12$ in parallel with Theorem~\ref{thm:1}.  

Our last theorem provides bounds for the higher order Sobolev norms of \eqref{sch}. In particular, we demonstrate how the smoothing estimates can be used to obtain polynomial-in-time growth bounds for the global solutions of equation \eqref{sch}.  This is a line of research initiated by Bourgain in \cite{boubounds} where he introduced a method for obtaining a priori bounds on Sobolev norms for equations lacking   suitable   conservation laws. 
 
\begin{theorem} \label{thm:n4} Consider the equation \eqref{sch} on $\T$.  Fix $\frac12<\alpha\leq 1 ,$ and   $s> \alpha$. We have 
$$\|u(t)\|_{H^s}\les  \la t\ra^{\max(1,\frac{s-\alpha}{2\alpha-1})}.$$
On $\R$ a similar inequality  inequality holds with an $\epsilon$ loss:
$$\|u(t)\|_{H^s}\les  \la t\ra^{\max(1,\frac{s-\alpha}{2\alpha-1}+)}.$$
\end{theorem} 

\vskip 0.05 in
\noindent

\vskip 0.05in
\noindent

The paper is organized as follows.  In Section 2 we introduce our notation and define the spaces that we use. In addition we state two elementary lemmas that we use in order to prove the multilinear smoothing estimates. Section 3 introduces the normal form method for the periodic fractional Schr\"odinger equation. At the end of the section we provide the proofs of the main theorems of our paper.  Section 4 contains the smoothing estimates for the fractional NLS on the real line.  Finally on an appendix we provide a new smoothing estimate for the $L^2$-critical quintic NLS on the real line. In particular, we prove that for any $s>\frac12$ the nonlinear part of a solution is half a derivative smoother than the initial data. This straightforward result seems to be new and in particular improves the 1d smoothing result in \cite{kervar}. This latter paper uses bilinear refinement of Strichartz estimates in the spirit of \cite{jbou} to obtain global smoothing estimates in all dimensions for a certain class of semilinear Schr\"odinger equations.

\section{Notation and Preliminaries}
Recall that for $s\geq 0$, $H^s(\T)$ is defined as a subspace of $L^2$ via the norm
$$
\|f\|_{H^s(\T)}:=\sqrt{\sum_{k\in\Z} \la k\ra^{2s} |\widehat{f}(k)|^2},
$$
where $\la k\ra:=(1+k^2)^{1/2}$ and $\widehat{f}(k)=\frac{1}{2\pi}\int_0^{2\pi}f(x)e^{-ikx} dx$ are the Fourier coefficients of $f$. Plancherel's theorem takes the form
$$\sum_{k\in\Z}  |\widehat{f}(k)|^2=\frac{1}{2\pi}\int_{0}^{2\pi}|f(x)|^2dx.$$
We denote the linear propagator of the Schr\"odinger equation as $e^{it(-\Delta)^\alpha}$, where it is defined on the Fourier side as $\widehat{(e^{it(-\Delta)^\alpha}f)}(n)=e^{ it n^{2\alpha}}\widehat{f}(n)$. Similarly, $|\nabla|^{\alpha}$ is defined as $\widehat{|\nabla|^{\alpha}f)}(n)=  n^{ \alpha} \widehat{f}(n)$.
We   also use $(\cdot)^+$ to denote $(\cdot)^\epsilon$ for all $\epsilon>0$ with implicit constants depending on $\epsilon$.

  The Bourgain spaces, $X^{s,b}$, will be defined as the closure of compactly supported smooth functions under the norm $$\|u\|_{X^{s,b}} :=\|e^{-it(-\Delta)^\alpha}u\|_{H^b_t(\mathbb{R})H^s_x(\mathbb{T})}=\|\langle \tau-|n|^{2\alpha} \rangle^{b} \langle n \rangle^s\widehat{u}(n,\tau)\|_{L_{\tau}^2 l^2_{n}}.$$ 
  On the real line, the definitions are similar.

We close this section by presenting two elementary lemmas that will be used repeatedly. For the proof of the first lemma see  the Appendix of \cite{erdtzi1}. For the second lemma, see \cite{det}, where it was proved for $m,n,k\in \Z$. The proof remains valid in the case that $m,n,k\in \R$. 
\begin{lemma}\label{lem:sums}   If  $\beta\geq \gamma\geq 0$ and $\beta+\gamma>1$, then
\be\nn
\sum_n\frac{1}{\la n-k_1\ra^\beta \la n-k_2\ra^\gamma}\lesssim \la k_1-k_2\ra^{-\gamma} \phi_\beta(k_1-k_2),
\ee
and 
\be\nn
\int_\R \frac{1}{\la \tau-k_1\ra^\beta \la \tau-k_2\ra^\gamma} d\, \tau \lesssim \la k_1-k_2\ra^{-\gamma} \phi_\beta(k_1-k_2),
\ee
where
 \be\nn
\phi_\beta(k):=\sum_{|n|\leq |k|}\frac1{\la n\ra^\beta}\sim \left\{\begin{array}{ll}
1, & \beta>1,\\
\log(1+\la k\ra), &\beta=1,\\
\la k \ra^{1-\beta}, & \beta<1.
 \end{array}\right.
\ee
\end{lemma}

\begin{lemma}\label{freq_est} Fix $\alpha\in (1/2,1]$. For $m,n ,k\in \R$, we have 
  $$g(m,n,k):=|(k+n)^{2\alpha}-(k+m+n)^{2\alpha}+(k+m)^{2\alpha}-k^{2\alpha}|\ges \frac{|m||n|}{(|m|+|n|+|k|)^{2-2\alpha}},$$
  where the implicit constant depends on $\alpha$.   
\end{lemma}

\section{Fractional NLS on the Torus}
As in \cite{et} we will apply a normal form transform on the Fourier side. 
After the change of variable $u(x,t)\to u(x,t) e^{iPt}$, where $P=\frac{1}{\pi}\|u_0\|_2^2$,   equation \eqref{sch}
becomes
\be\label{wickNLS}
iu_t+(-\Delta)^\alpha u-|u|^2u+Pu=0,\,\,\,\,\,\, t\in \R,\,\,\,x\in\T,
\ee
with initial data in $u_0\in H^s(\T)$, $s>0$. 

Note the following identity which follows from Plancherel's theorem:
\begin{multline}\label{res_decomp}
\widehat{|u|^2u}(k)=\sum_{k_1,k_2 }   \widehat u(k_1)\overline{\widehat u(k_2)}\widehat u(k-k_1+k_2)\\ =\frac1\pi
\|u\|_2^2\widehat{u}(k)-|\widehat u(k)|^2\widehat{u}(k)+\sum_{k_1\neq k, k_2\neq k_1}  \widehat u(k_1)\overline{\widehat u(k_2)}\widehat u(k-k_1+k_2).
\end{multline} 
Using \eqref{res_decomp}, and letting $v_k(t):=e^{itk^{2\alpha}}u(\widehat k, t)$, we rewrite the equation \eqref{wickNLS} as
\begin{multline}\label{eq:v_normal}
i\partial_t v_k=-|v_k|^2v_k+\sum_{k-k_1+k_2-k_3=0 \atop{k_1\neq k, k_2\neq k_1}} e^{it(k^{2\alpha}-k_1^{2\alpha} +k_2^{2\alpha}-k_3^{2\alpha})} v_{k_1}\overline{v_{k_2}}v_{k_3}\\
= -|v_k|^2v_k+\sum_{k-k_1+k_2-k_3=0 \atop{0<|k_1-k| |k_2-k_1|\ll \la k\ra^{2-2\alpha}}} e^{it(k^{2\alpha}-k_1^{2\alpha} +k_2^{2\alpha}-k_3^{2\alpha})} v_{k_1}\overline{v_{k_2}}v_{k_3} \\  + \sum_{k-k_1+k_2-k_3=0 \atop{|k_1-k| |k_2-k_1|\ges \la k\ra^{2-2\alpha}}} e^{it(k^{2\alpha}-k_1^{2\alpha} +k_2^{2\alpha}-k_3^{2\alpha})} v_{k_1}\overline{v_{k_2}}v_{k_3}.
\end{multline}
Note that the  sum in the second line above is zero when $\alpha=1$. 

We now apply a normal form transform to the last sum. By Lemma~\ref{freq_est}, for $k,m,n\in \Z\setminus\{0\}$, we have 
$$
g(m,n,k) \ges 1
$$
if $|mn| \ges |k|^{2-2\alpha}$. Therefore, the phase in the last sum is $\ges 1$. Writing
$$
e^{it(k^{2\alpha}-k_1^{2\alpha} +k_2^{2\alpha}-k_3^{2\alpha})}=-i \frac{ \partial_t e^{it(k^{2\alpha}-k_1^{2\alpha} +k_2^{2\alpha}-k_3^{2\alpha})}}{k^{2\alpha}-k_1^{2\alpha} +k_2^{2\alpha}-k_3^{2\alpha}}
$$
we have
\begin{multline} \label{eq:NR}
\sum_{k-k_1+k_2-k_3=0 \atop{|k_1-k| |k_2-k_1|\ges \la k\ra^{2-2\alpha}}} e^{it(k^{2\alpha}-k_1^{2\alpha} +k_2^{2\alpha}-k_3^{2\alpha})} v_{k_1}\overline{v_{k_2}}v_{k_3}  \\
=-i \partial_t \Big(\sum_{k-k_1+k_2-k_3=0 \atop{|k_1-k| |k_2-k_1|\ges \la k\ra^{2-2\alpha}}} \frac{e^{it(k^{2\alpha}-k_1^{2\alpha} +k_2^{2\alpha}-k_3^{2\alpha})}}{k^{2\alpha}-k_1^{2\alpha} +k_2^{2\alpha}-k_3^{2\alpha}} v_{k_1}\overline{v_{k_2}}v_{k_3}\Big)\\
 +i  \sum_{k-k_1+k_2-k_3=0 \atop{|k_1-k| |k_2-k_1|\ges \la k\ra^{2-2\alpha}}} \frac{e^{it(k^{2\alpha}-k_1^{2\alpha} +k_2^{2\alpha}-k_3^{2\alpha})}}{k^{2\alpha}-k_1^{2\alpha} +k_2^{2\alpha}-k_3^{2\alpha}} \partial_t (v_{k_1}\overline{v_{k_2}}v_{k_3}). 
\end{multline}
In the last sum, by symmetry, it suffices to consider the cases when the derivative hits $v_{k_2}$ and $v_{k_3}$.  Using the equation \eqref{eq:v_normal}, we can rewrite this sum as   
\begin{multline*} 
  2\sum_{k-k_1+k_2-k_3+k_4-k_5=0 \atop{|k_1-k| |k_2-k_1|\ges \la k\ra^{2-2\alpha}}} \frac{e^{it(k^{2\alpha}-k_1^{2\alpha} +k_2^{2\alpha}-k_3^{2\alpha}+k_4^{2\alpha}-k_5^{2\alpha})}}{k^{2\alpha}-k_1^{2\alpha} +k_2^{2\alpha}-(k-k_1+k_2)^{2\alpha}} v_{k_1}\overline{v_{k_2}}v_{k_3}\overline{v_{k_4}}v_{k_5} \\
- \sum_{k-k_1+k_2-k_3+k_4-k_5=0 \atop{|k_1-k| |k-k_3|\ges \la k\ra^{2-2\alpha}}} \frac{e^{it(k^{2\alpha}-k_1^{2\alpha} +k_2^{2\alpha}-k_3^{2\alpha}+k_4^{2\alpha}-k_5^{2\alpha})}}{k^{2\alpha}-k_1^{2\alpha} +(k-k_1-k_3)^{2\alpha}-k_3^{2\alpha}} v_{k_1}\overline{v_{k_2}}v_{k_3}\overline{v_{k_4}}v_{k_5} \\
+ \sum_{k-k_1+k_2-k_3=0 \atop{|k_1-k| |k_2-k_1|\ges \la k\ra^{2-2\alpha}}} \frac{e^{it(k^{2\alpha}-k_1^{2\alpha} +k_2^{2\alpha}-k_3^{2\alpha})}}{k^{2\alpha}-k_1^{2\alpha} +k_2^{2\alpha}-k_3^{2\alpha}} v_{k_1}\overline{v_{k_2}}v_{k_3}(|v_{k_2}|^2-2|v_{k_3}|^2). 
\end{multline*} 
Using this formula in \eqref{eq:NR}  and \eqref{eq:v_normal}, and unraveling the change of variable, we have 
\be\label{preduhamel}
i\partial_t \left(e^{it(-\Delta)^\alpha} u + e^{it(-\Delta)^\alpha}  B(u) \right)=  e^{it(-\Delta)^\alpha} \big(R(u) + NR_1(u) +NR_2(u) +NR_3(u)\big),
\ee
where
$$
\widehat{B(u)}(k)= \sum_{k-k_1+k_2-k_3=0 \atop{|k_1-k| |k_2-k_1|\ges \la k\ra^{2-2\alpha}}} \frac{u_{k_1}\overline{u_{k_2}}u_{k_3}}{k^{2\alpha}-k_1^{2\alpha} +k_2^{2\alpha}-k_3^{2\alpha}}  ,
$$
$$
\widehat{R(u)}(k)=-|u_k|^2u_k+\sum_{k-k_1+k_2-k_3=0 \atop{0<|k_1-k| |k_2-k_1|\ll \la k\ra^{2-2\alpha}}} u_{k_1}\overline{u_{k_2}}u_{k_3},
$$
$$
\widehat{NR_1(u)}(k)=2\sum_{k-k_1+k_2-k_3+k_4-k_5=0 \atop{|k_1-k| |k_2-k_1|\ges \la k\ra^{2-2\alpha}}} \frac{u_{k_1}\overline{u_{k_2}}u_{k_3}\overline{u_{k_4}}u_{k_5}}{k^{2\alpha}-k_1^{2\alpha} +k_2^{2\alpha}-(k-k_1+k_2)^{2\alpha}}, 
$$
$$
\widehat{NR_2(u)}(k)=- \sum_{k-k_1+k_2-k_3+k_4-k_5=0 \atop{|k_1-k| |k-k_3|\ges \la k\ra^{2-2\alpha}}} \frac{u_{k_1}\overline{u_{k_2}}u_{k_3}\overline{u_{k_4}}u_{k_5} }{k^{2\alpha}-k_1^{2\alpha} +(k-k_1-k_3)^{2\alpha}-k_3^{2\alpha}}, 
$$
$$
\widehat{NR_3(u)}(k)= \sum_{k-k_1+k_2-k_3=0 \atop{|k_1-k| |k_2-k_1|\ges \la k\ra^{2-2\alpha}}} \frac{u_{k_1}\overline{u_{k_2}}u_{k_3}(|u_{k_2}|^2-2|u_{k_3}|^2)}{k^{2\alpha}-k_1^{2\alpha} +k_2^{2\alpha}-k_3^{2\alpha}}. 
$$
The following proposition provides Sobolev space estimates for the terms appearing in \eqref{preduhamel}. These bounds are sufficient for the proof of Theorem~\ref{thm:1} in the case $ s>\frac12$. For smaller $s$, we will need $X^{s,b}$ spaces, see Proposition~\ref{prop:main2} below.  
\begin{prop}\label{prop:main} Fix $\frac12<\alpha\leq 1.$ For any $s> \frac34-\frac\alpha2 $, and $a\leq \min(2\alpha-1,2s+\alpha-1) $ (the inequality has to be strict if the minimum is  $2s+\alpha-1$), we have  
$$
\|B(u)\|_{H^{s+a}}\les \|u\|_{H^s}^3,
$$
$$
\|R(u)\|_{H^{s+a}}\les \|u\|_{H^s}^3,
$$
$$
\|NR_3(u)\|_{H^{s+a}}\les \|u\|_{H^s}^5.
$$
If  $s>\frac12$, then for any $a\leq 2\alpha-1$, we have
$$
\|NR_j(u)\|_{H^{s+a}}\les \|u\|_{H^s}^5,\,\,\,\,j=1,2.
$$
\end{prop}
\begin{proof}
$$
\|B(u)\|_{H^{s+a}}^2\les \Big\|\sum_{k-k_1+k_2-k_3=0 \atop{|k_1-k| |k_2-k_1|\ges \la k\ra^{2-2\alpha}}} \frac{ \la k\ra^{s+a}|u_{k_1}| |u_{k_2}|   |u_{k_3}| }{|k^{2\alpha}-k_1^{2\alpha} +k_2^{2\alpha}-k_3^{2\alpha}|}  \Big\|_{\ell^2_k}^2.
$$
By Cauchy Schwarz inequality in  the sum we bound this by
$$
\|u\|_{H^s}^6 \sup_k \sum_{k-k_1+k_2-k_3=0 \atop{|k_1-k| |k_2-k_1|\ges \la k\ra^{2-2\alpha}}} \frac{ \la k\ra^{2s+2a} \la k_1\ra^{-2s} \la k_2\ra^{-2s} \la k_3\ra^{-2s}   }{|k^{2\alpha}-k_1^{2\alpha} +k_2^{2\alpha}-k_3^{2\alpha}|^2}.  
$$
Renaming the variables $k_1=k+m$, $k_2=k+m+n$, $k_3=k+n$,  we  rewrite the supremum above as
\be\label{Bestimate}
\sup_k \sum_{ |mn|\ges \la k\ra^{2-2\alpha} } \frac{ \la k\ra^{2s+2a} \la k+m\ra^{-2s} \la k+m+n\ra^{-2s} \la k+n\ra^{-2s}   }{g(m,n,k)^2}.  
\ee
Using Lemma~\ref{freq_est}, we bound the sum above by 
\be \label{temp1}
\sum_{|mn|\ges \la k\ra^{2-2\alpha} \atop{|m|\geq |n|}} \frac{ \la k\ra^{2s+2a}  (|m|  +|k|)^{4-4\alpha}}{ m^2 n^2   \la k+m\ra^{ 2s} \la k+m+n\ra^{ 2s} \la k+n\ra^{ 2s} },  
\ee
which we estimate in three separate regions. \\
Case 1. $|m|\gg |k|$. The sum is  
  \begin{align*} 
  \les   \sum_{ |m|\geq |n|>0} \frac{ \la k\ra^{2s+2a}  |m|^{2-4\alpha-2s}}{  n^2   \la k+m+n\ra^{ 2s} \la k+n\ra^{ 2s} }.
  \end{align*}
  For $s>\frac12$, we replace $m$ with $k$ in the numerator, and sum in $m$ and then in $n$ (using Lemma~\ref{lem:sums}) to obtain
    \begin{align*} 
 \les   \sum_{ |m|\geq |n|>0} \frac{ \la k\ra^{2a+2-4\alpha}  }{  n^2   \la k+m+n\ra^{ 2s} \la k+n\ra^{ 2s} } \les   \sum_{   |n|>0} \frac{ \la k\ra^{ 2a+2-4\alpha }}{  n^2     \la k+n\ra^{ 2s} } \les \la k\ra^{2a+2-4\alpha-\min(2s,2)} ,
  \end{align*}
 which is bounded in $k$ provided that $a\leq 2\alpha-1+\min(s,1).$ 
 
 For $0\leq s\leq \frac12$, the sum is bounded by 
$$
 \sum_{ |m|\geq |n|>0} \frac{ \la k\ra^{2a+3-4\alpha-2s+\epsilon}}{  n^2   \la k+m+n\ra^{ 2s}  |m|^{1-2s+\epsilon}\la k+n\ra^{ 2s} }  
  \les  \sum_{  |n|>0} \frac{ \la k\ra^{2a+3-4\alpha-2s+\epsilon}}{  n^2   \la k+n\ra^{ 2s+\epsilon} } 
  \les  \la k\ra^{2a+3-4\alpha-4s },
  $$    
  which is bounded in $k$ provided that $a\leq 2s+2\alpha-\frac32.$ 
  
\noindent
Case 2. $|m|\approx |k|$. In this region we have the bound
  $$
\sum_{ |k|\approx |m|\geq |n|>0} \frac{ \la k\ra^{2s+2a+2-4\alpha}}{  n^2   \la k+m\ra^{ 2s} \la k+m+n\ra^{ 2s} \la k+n\ra^{ 2s} } 
\les \sum_{  |n|>0} \frac{ \la k\ra^{2s+2a+2-4\alpha} \mathcal{A}}{  n^2    \la k+n\ra^{ 2s} }, 
$$
  where $\mathcal A = |n|^{-2s}$ if $s>\frac12$ and   $\mathcal A=1$ if $ \frac12\geq s>1/4$. 
This implies the bound   
$ \la k\ra^{ 2a+2-4\alpha }$, which is acceptable provided that $a\leq 2\alpha -1 .$
   
\noindent
Case 3. $|m|\ll |k|$. In this region we have the bound 
$$
\sum_{|mn|\ges \la k\ra^{2-2\alpha} \atop{0<|n|\leq |m|\ll|k|} } \frac{ \la k\ra^{ 2a-4s +4-4\alpha}   }{ m^2 n^2    } \les \sum_{ 0<|n| \ll|k|  } \frac{ \la k\ra^{ 2a-4s +2-2\alpha}   }{|n|   } \les \la k\ra^{ 2a-4s +2-2\alpha+}  .   
$$
This is bounded in $k$ provided that $a<2s+\alpha-1 $.

The bound for $NR_3(u)$ follows from the bound for $B(u)$ above since 
$$\sup_k |u_k|\leq \|u_k\|_{\ell^2}\les \|u\|_{H^s}.$$

Note that the first summand in the definition of $R(u)$ is in $H^{s+a}$ for $a\leq 2s$.  For the second summand in the definition of $R(u)$, using Cauchy Schwarz inequality,  it suffices to bound
$$\sup_k \sum_{|mn|\ll \la k\ra^{2-2\alpha} \atop{|m|\geq |n|>0}} \frac{ \la k\ra^{2s+2a}  }{    \la k+m\ra^{ 2s} \la k+m+n\ra^{ 2s} \la k+n\ra^{ 2s} }.   
$$
Since $\alpha >\frac12$, we have $|n|\leq |m|\ll |k|$ in the sum. Thus we can bound the sum by 
$$  \sum_{|mn|\ll \la k\ra^{2-2\alpha} } \la k\ra^{ 2a-4s} \les \la k\ra^{ 2a-4s+2-2\alpha+}.  
$$
This is bounded in $k$ provided that $a<2s+\alpha-1$.

By a similar change of variable, to estimate $NR_1(u)$, we need to consider
$$
\sup_k   \sum_{k+n-k_3+k_4-k_5=0 \atop{|m| |n|\ges \la k\ra^{2-2\alpha}}}  \frac{\la k \ra^{2s+2a} }{g(m,n,k)^2   \la k+m\ra^{ 2s} \la k+m+n\ra^{ 2s} \la k_3\ra^{ 2s}\la k_4\ra^{ 2s}\la k_5\ra^{ 2s} }.
$$
Summing in $k_4$ and $k_3$ for $s>\frac12$, we bound this by
$$
\sup_k \sum_{ |m| |n|\ges \la k\ra^{2-2\alpha} }\frac{\la k \ra^{2s+2a}  }{g(m,n,k)^2 \la k+m\ra^{ 2s} \la k+m+n\ra^{ 2s} \la k+n\ra^{ 2s}   }, 
$$
which was considered before, see \eqref{Bestimate}.   

Analogously, for $NR_2(u)$, the multiplier we need to bound is
$$
\sup_k \sum_{k_2-(k+m+n) +k_4-k_5=0 \atop{|m| |n|\ges \la k\ra^{2-2\alpha}}} \frac{\la k \ra^{2s+2a} }{g(m,n,k)^2 \la k+m\ra^{ 2s} \la k +n\ra^{ 2s} \la k_2\ra^{ 2s}\la k_4\ra^{ 2s}\la k_5\ra^{ 2s}  }.
$$
Summing in $k_4$ and $k_2$, we reduce  this case to   \eqref{Bestimate}, too.
\end{proof}
The following proposition, which follows from the proof of Proposition~\ref{prop:main}  in a fairly simple way, will be used in the proof of Theorem~\ref{thm:2}.
\begin{prop}\label{prop:mainvar} Fix $\frac12<\alpha\leq 1$, and $s_0>\frac12$. For any $s\geq s_0$, and $a\leq  2\alpha-1 $, we have  
$$
\|B(u)\|_{H^{s+a}}\les \|u\|_{H^s}\|u\|_{H^{s_0}}^2,
$$
$$
\|R(u)\|_{H^{s+a}}\les \|u\|_{H^s}\|u\|_{H^{s_0}}^2,
$$
$$
\|NR_j(u)\|_{H^{s+a}}\les \|u\|_{H^s}\|u\|_{H^{s_0}}^4,\,\,\,\,j=1,2,3.
$$
\end{prop} 
\begin{proof}  We describe the details for $B$ only. 
$$
\|B(u)\|_{H^{s+a}}^2\les \Big\|\sum_{k-k_1+k_2-k_3=0 \atop{|k_1-k| |k_2-k_1|\ges \la k\ra^{2-2\alpha}}} \frac{ \la k\ra^{s+a}|u_{k_1}| |u_{k_2}|  |u_{k_3}| }{|k^{2\alpha}-k_1^{2\alpha} +k_2^{2\alpha}-k_3^{2\alpha}|}  \Big\|_{\ell^2_k}^2.
$$
Since $\max(|k_1|,|k_2|,|k_3|)\gtrsim |k|$, we can bound the right hand side  by
$$
\sum_{j=1}^3 \Big\|\sum_{k-k_1+k_2-k_3=0 \atop{|k_1-k| |k_2-k_1|\ges \la k\ra^{2-2\alpha}}} \frac{ \la k\ra^{s_0+a}   \la k_j\ra^{s-s_0} |u_{k_1}| |u_{k_2}| |u_{k_3}| }{|k^{2\alpha}-k_1^{2\alpha} +k_2^{2\alpha}-k_3^{2\alpha}|}  \Big\|_{\ell^2_k}^2.
$$
After Cauchy Schwarz inequality, the sum above is majorized by 
$$
\|u\|_{H^s}^2 \|u\|_{H^{s_0}}^4 \sup_k \sum_{k-k_1+k_2-k_3=0 \atop{|k_1-k| |k_2-k_1|\ges \la k\ra^{2-2\alpha}}} \frac{ \la k\ra^{2s_0+2a} \la k_1\ra^{-2s_0} \la k_2\ra^{-2s_0} \la k_3\ra^{-2s_0}   }{|k^{2\alpha}-k_1^{2\alpha} +k_2^{2\alpha}-k_3^{2\alpha}|^2}.   
$$
Implementing  the smoothing bound we have in Proposition~\ref{prop:main} at the $ s_0$ level, we are done. 

The proofs for $R(u)$ and $NR_j(u)$  are similar. 
\end{proof}

\begin{prop}\label{prop:main2} Fix $\frac12<\alpha\leq 1.$ For any $\frac34-\frac{\alpha}2<s\leq  \frac12, $ , and $a\leq \min(2\alpha-1,2s+\alpha-1) $ (the inequality has to be strict if the minimum is  $2s+\alpha-1$), we have   
$$
\|NR_j(u)\|_{X^{s+a,-b}}\les \|u\|_{X^{s,b}}^5,\,\,\,\,j=1,2,
$$
provided that $b<\frac12$ is sufficiently close to $\frac12$.  
\end{prop}
Before the proof of Proposition~\ref{prop:main2} we record the following two lemmas.
  
\begin{lemma} \label{lem:231}For $ \frac14<s\leq \frac12$, we have
$$
\sum_{k=k_1-k_2+k_3} \frac{\la k_1\ra^{-2s}\la k_2\ra^{-2s} \la k_3\ra^{-2s}  }{ \la A -k_1^2+k_2^2-k_3^2\ra^{1-} } \les \la k\ra^{-2s}
$$
where the implicit constant is independent of $k$ and $A.$
\end{lemma}
\begin{proof}
Renaming the variables $k_1=k+m$, $k_2=k+m+n$, $k_3=k+n$, and replacing $A$ with $A+k^2$, and by symmetry, it suffices to consider 
$$
\sum_{|m|\geq |n|} \frac{\la k+m\ra^{-2s}\la k+m+n\ra^{-2s} \la k+n\ra^{-2s}  }{ \la A + 2mn\ra^{1-} }. 
$$
Noting that the term $n=0$ is $\les \la k\ra^{-2s}$,  we  reduce the above sum to 
$$
\sum_{|m|\geq |n|>0} \frac{\la k+m\ra^{-2s}\la k+m+n\ra^{-2s} \la k+n\ra^{-2s}  }{ \la A + 2mn\ra^{1-} }, 
$$
and treat this sum in two distinct cases. \\ 
Case 1. $|A+2mn|\les |n|$. We have
$$
\sum_{|m|\geq |n|>0, \,|A+2mn|\les |n|}  \la k+m\ra^{-2s}\la k+m+n\ra^{-2s} \la k+n\ra^{-2s}. 
$$
Note that  $\max( |k+m+n|,|k+n|,|k+m|)\gtrsim |k|$. In the case $|k+m+n|\gtrsim |k|$, by Cauchy Schwarz inequality, we have
$$
\la k\ra^{-2s} \sqrt{\sum_{|m|\geq |n|>0, \,|A+2mn|\les |n|}  \la k+m\ra^{-4s}} \sqrt{ \sum_{|m|\geq |n|>0, \,|A+2mn|\les |n|}  \la k+n\ra^{-4s}} \les \la k\ra^{-2s}. 
$$
The last inequality follows by observing that for fixed $m$ there are only finitely many $n$ in the sum with the property $|A+2mn|\les |n|$, and vice versa. The other cases are similar. \\
Case 2. $|A+2mn|\gg |n|$. We have 
$$
\sum_{|m|\geq |n|>0} \frac{\la k+m\ra^{-2s}\la k+m+n\ra^{-2s} \la k+n\ra^{-2s}  }{  \la n\ra^{1-} } 
\les \sum_n \frac{  \la k+n\ra^{-2s}  }{  \la n\ra^{4s-} } \les \la k\ra^{-2s}. 
$$
This finishes the proof of the lemma.
\end{proof}
\begin{lemma} \label{lem:23alpha}Fix $\frac12<\alpha<1$. For $ \frac34-\frac\alpha2<s\leq \frac12$, we have
$$
\sum_{k=k_1-k_2+k_3} \frac{\la k_1\ra^{-2s}\la k_2\ra^{-2s} \la k_3\ra^{-2s}  }{ \la A \pm k_1^{2\alpha}\pm k_2^{2\alpha} \pm k_3^{2\alpha} \ra^{1-} } \les \la k\ra^{-2s}
$$
where the implicit constant is independent of $k, A$, and the combinations of $\pm$ signs.
\end{lemma}
\begin{proof}Replacing $k_2$ with $-k_2$ we need to prove that 
\be\label{alpha2s}
\sum_{k=k_1+k_2+k_3} \frac{\la k_1\ra^{-2s}\la k_2\ra^{-2s} \la k_3\ra^{-2s}  }{ \la A \pm k_1^{2\alpha}\pm k_2^{2\alpha} \pm k_3^{2\alpha} \ra^{1-} } \les \la k\ra^{-2s}.
\ee
By symmetry we can  assume that $|k_3|\geq |k_2|\geq |k_1|$. We have the following regions\\
1. $|k_3|\approx |k_2|\gg |k| \gg |k_1|$,\\
2.  $|k_3|\approx |k_2|\gg |k_1| \gtrsim \la k\ra $,\\
3.  $|k_3|\approx |k_2|\approx  |k_1| \gg \la k\ra$,\\
4.  $|k| \approx |k_3|\gg |k_2|\geq |k_1|$,\\
5.  $|k|\approx |k_3|\approx |k_2|\geq   |k_1|$.\\
Case 1. $|k_3|\approx |k_2|\gg |k| \gg |k_1| $.  We have
$$
\eqref{alpha2s}\les \sum_{|k_3| \gg |k| \gg |k_1| } \frac{\la k_1\ra^{-2s}  \la k_3\ra^{-4s}  }{ \la f(k_1,k_3) \ra^{1-} } \les \sum_{\ell=3}^\infty 2^{-4\ell s} | k |^{-4s} \sum_{ |k_1|\ll |k|} \la k_1\ra^{-2s} \sum_{|k_3| \approx    2^\ell |k| } \frac{ 1  }{ \la f(k_1,k_3) \ra^{1-} } , 
$$
where $f(k_1,k_3)= A \pm k_1^{2\alpha}\pm (k-k_1-k_3)^{2\alpha} \pm k_3^{2\alpha} $.
Note that,  for fixed $\ell $ and $k_1$,
$$\Big|\frac{\partial f}{\partial k_3}\Big| \gtrsim | |k-k_1-k_3|^{2\alpha-1} -  |k_3|^{2\alpha-1}| \approx   |k|   |k_3|^{2\alpha-2}\approx |k|^{2\alpha-1} 2^{\ell(2\alpha-2)}.$$
Therefore by the mean value theorem, there is a $k_0$ depending on $A, k, k_1, \ell$ and $\pm$ signs so that the sum in $k_3$ is bounded by 
$$
\sum_{|k_3| \approx    2^\ell |k| } \frac1{ \la |k_3-k_0| |k|^{2\alpha-1} 2^{\ell(2\alpha-2)} \ra^{1-}} \les 1+ |k|^{1-2\alpha+} 2^{\ell(2-2\alpha)+}.
$$
Also summing in $k_1$ we have the bound
$$
\eqref{alpha2s}  \les \sum_{\ell=3}^\infty 2^{-4\ell s}| k |^{1-6s+} (1+ |k|^{1-2\alpha+} 2^{\ell(2-2\alpha)+}) \les  | k |^{1-6s+}+ |k|^{2-6s-2\alpha+}\les |k|^{-2s}. 
$$
Case 2. $|k_3|\approx |k_2|\gg |k_1| \gtrsim \la k\ra $. We have
$$
\eqref{alpha2s}\les \sum_{|k_3| \gg |k_1| \gtrsim \la k\ra  } \frac{\la k\ra^{-2s}  \la k_3\ra^{-4s}  }{ \la f(k_1,k_3) \ra^{1-} },
$$
where $f$ is as above. Note that for fixed $k_3$
$$\Big|\frac{\partial f}{\partial k_1}\Big| \gtrsim \big|  |k-k_1-k_3|^{2\alpha-1} -  |k_1|^{2\alpha-1}\big| \approx     |k_3|^{2\alpha-1}.$$
Therefore, for some $k_0$ depending $A, k, k_3, \ell$ and $\pm$ signs, we have 
$$
\eqref{alpha2s}\les \sum_{|k_3| \gg   \la k\ra  }  \sum_{|k_1|\ll |k_3|} \frac{\la k\ra^{-2s}  \la k_3\ra^{-4s}  }{ \la  |k_1-k_0| |k_3|^{2\alpha-1} \ra^{1-} } \les \sum_{|k_3| \gg  \la k\ra  } \la k\ra^{-2s}  \la k_3\ra^{-4s}  (1+ |k_3|^{1-2\alpha+}) \les \la k\ra^{-2s}.
$$ 
Case 3. $|k_3|\approx |k_2|\approx  |k_1| \gg\la k\ra$. We have
$$
\eqref{alpha2s}\les \sum_{|k_3| \gg \la k\ra  }  |k_3|^{-6s}  \sum_{|k_1|\approx |k_3|} \frac1{ \la f(k_1,k_3) \ra^{1-} },
$$
where 
$$\Big|\frac{\partial f}{\partial k_1}\Big| \gtrsim \big|  |k-k_1-k_3|^{2\alpha-1} -  |k_1|^{2\alpha-1}\big|.$$
Note that in this region the derivative is zero only when $k_1=\frac{k-k_3}2$.  Since the contribution of this case can be bounded by $\la k\ra^{1-6s}\les \la k\ra^{-2s}$, it suffices to estimate 
$$
 \sum_{|k_3| \gg  \la k\ra  }  |k_3|^{-6s}  \sum_{\ell=1}^{\log|k_3|} \sum_{\big|k_1-\frac{k_3-k}2 \big| \approx 2^\ell,} \frac1{ \la f(k_1,k_3) \ra^{1-} }.
$$
Writing $j=k_1-\frac{k_3-k}2$, we have 
$$
\Big|\frac{\partial f}{\partial k_1}\Big| \gtrsim |j| |k_3|^{ 2\alpha-2}\approx  2^\ell  |k_3|^{ 2\alpha-2}.
$$
This yields (for some $j_0$ depending on $k_3$)
\begin{multline*}
\eqref{alpha2s}\les  \sum_{|k_3| \gg \la k\ra }  |k_3|^{-6s}  \sum_{\ell=1}^{\log|k_3|} \sum_{ |j| \approx 2^\ell,} \frac1{ \la |j-j_0| 2^\ell |k_3|^{ 2\alpha-2} \ra^{1-} } \\
\les \sum_{|k_3| \gg \la k\ra }  |k_3|^{-6s+} + \sum_{|k_3| \gg \la k\ra  }  |k_3|^{-6s+2-2\alpha+}\les \la k\ra^{1-6s+} +\la k\ra^{-6s+3-2\alpha +} \les  \la k\ra^{-2s},
\end{multline*}
provided that $s> \frac34-\frac\alpha2$.\\
Case 4.  $|k| \approx |k_3|\gg |k_2|\geq |k_1| $. It suffices to prove that
\be\label{case4}
 \sum_{|k|  \gg |k_2|\geq |k_1| } \frac{\la k_1\ra^{-2s}  \la k_2\ra^{-2s}  }{ \la f(k_1,k_2) \ra^{1-} } \les 1, 
\ee
where $f(k_1,k_2)= A \pm k_1^{2\alpha}\pm k_2^{2\alpha} \pm (k-k_1-k_2)^{2\alpha} $.
Note that
$$
\Big|\frac{\partial f}{\partial k_1}\Big| , \Big|\frac{\partial f}{\partial k_2}\Big| \gtrsim   |k |^{ 2\alpha-1}.
$$
Therefore 
$$
\eqref{case4}\les \sum_{|k|  \gg  |k_1| } \la k_1\ra^{-4s}  \sum_{|k|  \gg |k_2| }  \frac{  1 }{ \la |k_2-k_0| |k |^{ 2\alpha-1}\ra^{1-} }
\les \sum_{|k|  \gg  |k_1| } \la k_1\ra^{-4s} (1+|k|^{1- 2\alpha +})\les 1.
$$
Here, $k_0$ depends on $k_1$ as above.\\
Case 5.  $|k|\approx |k_3|\approx |k_2|\geq   |k_1|  $. We need to prove that
$$
|k|^{-2s} \sum_{|k|\gtrsim |k_1|}  \la k_1\ra^{-2s}   \sum_{|k_3|  \approx |k|} \frac{1 }{ \la f(k_1,k_3) \ra^{1-} } \les 1,
$$
where $f$ satisfies
$$\Big|\frac{\partial f}{\partial k_3}\Big| \gtrsim \big|  |k-k_1-k_3|^{2\alpha-1} -  |k_3|^{2\alpha-1}  \big|.$$
Note that the derivative is zero when $k_3=\frac{k-k_1}2$ or when $k_1=k$.  Since the contribution of these cases can be bounded by $|k|^{1-4s+}\les 1$, it suffices to estimate 
\be\label{case5}
 |k|^{-2s} \sum_{|k|\gtrsim |k_1|, k_1\neq k}  \la k_1\ra^{-2s}  \sum_{\ell=1}^{\log|k|} \sum_{\big|k_3-\frac{k -k_1}2 \big| \approx 2^\ell,} \frac1{ \la f(k_1,k_3) \ra^{1-} }.
\ee
Writing $j=k_3-\frac{k -k_1}2$, we have 
$$
\Big|\frac{\partial f}{\partial k_3}\Big| \gtrsim \min(2^\ell, |k-k_1|)  |k|^{ 2\alpha-2}.
$$
This yields (for some $j_0$ depending on $k_1$)
\begin{multline*}
\eqref{case5}\les  |k|^{-2s} \sum_{|k|\gtrsim |k_1|, k_1\neq k}  \la k_1\ra^{-2s}  \sum_{\ell=1}^{\log|k|} \sum_{|j| \approx 2^\ell } \frac1{ \la |j-j_0|\min(2^\ell, |k-k_1|)  |k|^{ 2\alpha-2} \ra^{1-} }  \\
\les |k|^{1-4s+}+ |k|^{-2s+2-2\alpha+} 
\sum_{|k|\gtrsim |k_1|, k_1\neq k}  \la k_1\ra^{-2s}  \sum_{\ell=1}^{\log|k|}   \frac1{ \min(2^\ell, |k-k_1|)^{1-}} \\ 
\les |k|^{1-4s+}+ |k|^{-4s+3-2\alpha+}\les 1,
\end{multline*}
provided that $s> \frac34-\frac\alpha2$.
\end{proof}

\begin{proof}[Proof of Proposition~\ref{prop:main2}] Recall that
$$
\widehat{NR_1(u)}(k)=2\sum_{k-k_1+k_2-k_3+k_4-k_5=0 \atop{|k_1-k| |k_2-k_1|\ges \la k\ra^{2-2\alpha}}} \frac{u_{k_1}\overline{u_{k_2}}u_{k_3}\overline{u_{k_4}}u_{k_5}}{k^{2\alpha}-k_1^{2\alpha} +k_2^{2\alpha}-(k-k_1+k_2)^{2\alpha}} .
$$
By the definition of $X^{s,b}$ norms and the Cauchy Schwarz inequality as  in \cite{det}, the multiplier we need to bound is
$$
\sup_k \sum_{k+n-k_3+k_4-k_5=0 \atop{|m| |n|\ges \la k\ra^{2-2\alpha}}} \frac{\la k \ra^{2s+2a} \la k+m\ra^{-2s} \la k+m+n\ra^{-2s} \la k_3\ra^{-2s}\la k_4\ra^{-2s}\la k_5\ra^{-2s} }{g(m,n,k)^2   \la k^{2\alpha}  -(k+m)^{2\alpha} +(k+m+n)^{2\alpha} -k_3^{2\alpha} +k_4^{2\alpha} -k_5^{2\alpha}\ra^{1-} }.
$$
Fix $k,m,n$. Using Lemma~\ref{lem:231} when $\alpha=1$ and    Lemma~\ref{lem:23alpha} when $\frac12<\alpha<1$ with $k=k+n$ and $A=  k^{2\alpha}  -(k+m)^{2\alpha} +(k+m+n)^{2\alpha}$, we bound the supremum above by
$$
\sup_k \sum_{ |m| |n|\ges \la k\ra^{2-2\alpha} } \frac{\la k \ra^{2s+2a} \la k+m\ra^{-2s} \la k+m+n\ra^{-2s} \la k+n\ra^{-2s}  }{g(m,n,k)^2 }.
$$
This sum is identical to \eqref{Bestimate} which was handled in Proposition~\ref{prop:main}.

Similarly, for $NR_2(u)$, the multiplier we need to bound is
$$
\sup_k \sum_{k_2-(k+m+n) +k_4-k_5=0 \atop{|m| |n|\ges \la k\ra^{2-2\alpha}}} \frac{\la k \ra^{2s+2a} \la k+m\ra^{-2s} \la k +n\ra^{-2s} \la k_2\ra^{-2s}\la k_4\ra^{-2s}\la k_5\ra^{-2s} }{g(m,n,k)^2   \la k^{2\alpha}  -(k+m)^{2\alpha} +k_2^{2\alpha}-(k +n)^{2\alpha}   +k_4^{2\alpha} -k_5^{2\alpha}\ra^{1-} }.
$$
This can also be bounded by \eqref{Bestimate} using  Lemma~\ref{lem:231} when $\alpha=1$ and   Lemma~\ref{lem:23alpha} when $\frac12<\alpha<1$ with $k=k+m+n$ and $A=  k^{2\alpha}  -(k+m)^{2\alpha} -(k+n)^{2\alpha}$.
\end{proof}

We are now ready to prove the smoothing theorems  and the theorem on growth bounds for Sobolev norms in the case of the  torus. 
\begin{proof}[Proof of Theorem~\ref{thm:2}] 
With the change of variable $u(x,t)\to u(x,t) e^{iPt}$ in the equation \eqref{sch}, where $P=\frac{1}{\pi}\|u_0\|_2^2$, it suffices to prove smoothing for the difference
$$
u(t)-e^{it(-\Delta)^\alpha}u_0.
$$
Note that $u$ satisfies \eqref{wickNLS}, and  hence \eqref{preduhamel}. 
Integrating \eqref{preduhamel} we obtain
\begin{multline}\label{nduhamel}
 u(t)-   e^{-it(-\Delta)^\alpha}u_0 = e^{-it(-\Delta)^\alpha} B(u_0)- B(u(t) ) \\
  -i\int_0^t e^{i(t^\prime-t)(-\Delta)^\alpha} \big[R(u(t^\prime)) + NR_1(u(t^\prime)) +NR_2(u(t^\prime)) +NR_3(u(t^\prime))\big] dt^\prime.
\end{multline}
Using Proposition~\ref{prop:mainvar} on the right hand side yields the claim. Continuity in time easily follows from the a priori estimates. 
The second part of the theorem follows from the first part and  the conservation law \eqref{conserve} at the  $H^\alpha$ level. 
\end{proof}
\begin{proof}[Proof of Theorem~\ref{thm:1}] 
This follows from Theorem~\ref{thm:2} for $s>\frac12$. 
For $s\leq\frac12$, we use $X^{s,b}$ spaces. Recall (see, e.g., Lemma~3.9 and Lemma~3.12 in \cite{etbook}) the 
continuous  embedding 
$$
X^{s,\frac12+}\subset C^0_tH^s_x,
$$ 
and 
$$\Big\|\int_0^t e^{i(-\Delta)^\alpha(t^\prime-t)}f(t^\prime)d t^\prime\Big\|_{X^{s,\frac12+}}\lesssim  \|f\|_{X^{s,-\frac12+}}.$$ 
Using Proposition~\ref{prop:main} and Proposition~\ref{prop:main2} 
on \eqref{nduhamel} (in the local existence interval) we  obtain
\begin{multline*} 
 \big\|u(t)-   e^{-it(-\Delta)^\alpha}u_0 \big\|_{H^{s+a}} \les  \|u_0\|_{H^s}^3+\|u(t)\|_{H^s}^3 +\int_0^t\big(\|u(t^\prime)\|_{H^s}^3+ \|u(t^\prime)\|_{H^s}^5\big)d t^\prime  \\
  +\Big\|\int_0^t e^{i(t^\prime-t)(-\Delta)^\alpha} \big(  NR_1(u(t^\prime)) +NR_2(u(t^\prime)) \big) dt^\prime\Big\|_{X^{s,\frac12+}}\\
  \les  \|u_0\|_{H^s}^3+\|u_0\|_{H^s}^5  
  +  \big\|  NR_1(u(t^\prime)) +NR_2(u(t^\prime))  \big \|_{X^{s,-\frac12+}}\\
  \les  \|u_0\|_{H^s}^3+\|u_0\|_{H^s}^5  + \|u\|_{X^{s,\frac12+}}^5   \les  \|u_0\|_{H^s}^3+\|u_0\|_{H^s}^5.  \qedhere
\end{multline*}
\end{proof}
 
 \begin{proof}[Proof of Theorem~\ref{thm:n4}] We prove the theorem only on $\T$. The proof on $\R$ is similar using Theorem~\ref{thm:n3} instead of Theorem~\ref{thm:2}.  
 
 First note that for $s\leq 3\alpha-1$, the claim follows from the second part of Theorem~\ref{thm:2}  with $s=\alpha$, and the boundedness of  the $H^\alpha$ norm.
  
Fix $s>3\alpha-1$. Using the second part of Theorem~\ref{thm:2} with  $a=2\alpha-1$ we have 
$$ \big\|u(t)-   e^{-it(-\Delta)^\alpha-itP}u_0 \big\|_{H^{s}}\les  
\|u_0\|_{H^{s-2\alpha+1}}   +\|u(t)\|_{H^{s-2\alpha+1}}   +\int_0^t \|u(t^\prime)\|_{H^{s-2\alpha+1}} dt^\prime,$$
which implies
$$ \big\|u(t) \big\|_{H^{s}}\les  
\|u_0\|_{H^s}   +\|u(t)\|_{H^{s-2\alpha+1}}    +\int_0^t \|u(t^\prime)\|_{H^{s-2\alpha+1}} dt^\prime.$$
By interpolation we have 
$$
\|u(t)\|_{H^{s-2\alpha+1}} \leq \|u(t)\|_{H^{\alpha}}^{\frac{2\alpha-1}{s-\alpha }}\|u(t)\|_{H^{s }}^{\frac{s-3\alpha+1}{s-\alpha }} \les  \|u(t)\|_{H^{s }}^{\frac{s-3\alpha+1}{s-\alpha }}.
$$
Therefore
 $$ \big\|u(t) \big\|_{H^{s}}\les  
\|u_0\|_{H^s}   +\|u(t)\|_{H^s}^{\frac{s-3\alpha+1}{s-\alpha }}    +\int_0^t \|u(t^\prime)\|_{H^s}^{\frac{s-3\alpha+1}{s-\alpha }} dt^\prime.$$
Letting $f(t):=\sup_{t^\prime\in [0,t]} \|u(t^\prime)\|_{H^s}$, we deduce 
$$
f(t)\les   \la t\ra f(t)^{\frac{s-3\alpha+1}{s-\alpha }},  
$$
which implies the claim.
\end{proof}

 \section{Fractional NLS on the Real Line} 
 We now apply a similar normal form transformation to the fractional NLS equation on the real line. The main difference with the torus case is that the phase rotation $P$ does not arise. The final equation reads as follows (compare with \eqref{preduhamel}) 
 \be\label{preduhamelR}
i\partial_t \left(e^{it(-\Delta)^\alpha} u + e^{it(-\Delta)^\alpha}  B(u) \right)=  e^{it(-\Delta)^\alpha} \big(R(u) + NR_1(u) +NR_2(u)  \big),
\ee
where
$$
\widehat{B(u)}(\xi)= \int\limits_{\xi-\xi_1+\xi_2-\xi_3=0 \atop{|\xi-\xi_1| |\xi-\xi_3|\ges (|\xi-\xi_1|+|\xi-\xi_3|+|\xi|)^{2-2\alpha}}} \frac{u_{\xi_1}\overline{u_{\xi_2}}u_{\xi_3}}{\xi^{2\alpha}-\xi_1^{2\alpha} +\xi_2^{2\alpha}- \xi_3^{2\alpha}},  
$$
$$
\widehat{R(u)}(\xi)=\int\limits_{\xi-\xi_1+\xi_2-\xi_3=0 \atop{|\xi-\xi_1| |\xi-\xi_3|\ll (|\xi-\xi_1|+|\xi-\xi_3|+|\xi|)^{2-2\alpha}}} u_{\xi_1}\overline{u_{\xi_2}}u_{\xi_3},
$$
$$
\widehat{NR_1(u)}(\xi)=2\int\limits_{\xi-\xi_1+\xi_2-\xi_3+\xi_4-\xi_5=0 \atop{|\xi-\xi_1| |\xi_2-\xi_1|\ges (|\xi-\xi_1|+|\xi_2-\xi_1|+|\xi|)^{2-2\alpha}}} \frac{u_{\xi_1}\overline{u_{\xi_2}}u_{\xi_3}\overline{u_{\xi_4}}u_{\xi_5}}{\xi^{2\alpha}-\xi_1^{2\alpha} +\xi_2^{2\alpha}-(\xi-\xi_1+\xi_2)^{2\alpha}}, 
$$
$$
\widehat{NR_2(u)}(\xi)=- \int\limits_{\xi-\xi_1+\xi_2-\xi_3+\xi_4-\xi_5=0 \atop{|\xi-\xi_1| |\xi-\xi_3|\ges (|\xi-\xi_1|+|\xi-\xi_3|+|\xi|)^{2-2\alpha}}} \frac{u_{\xi_1}\overline{u_{\xi_2}}u_{\xi_3}\overline{u_{\xi_4}}u_{\xi_5} }{\xi^{2\alpha}-\xi_1^{2\alpha} +(\xi-\xi_1-\xi_3)^{2\alpha}-\xi_3^{2\alpha}}. 
$$
\\
 The following proposition is analogous to Proposition ~\ref{prop:mainvar} for the torus case. Using this proposition one can prove Theorem~\ref{thm:n3} along the lines of the proof of Theorem~\ref{thm:2}.
\begin{prop}\label{prop:mainvarR} Fix $\frac12<\alpha\leq 1$, and $s_0>\frac12$. For any $s\geq s_0$, and $a < 2\alpha-1 $, we have  
$$
\|B(u)\|_{H^{s+a}}\les \|u\|_{H^s}\|u\|_{H^{s_0}}^2,
$$
$$
\|R(u)\|_{H^{s+a}}\les \|u\|_{H^s}\|u\|_{H^{s_0}}^2,
$$
$$
\|NR_j(u)\|_{H^{s+a}}\les \|u\|_{H^s}\|u\|_{H^{s_0}}^4,\,\,\,\,j=1,2.
$$
\end{prop}   
\begin{proof}
We will prove the  proposition for $s=s_0$ only. The statement for $s\geq s_0$ can be obtained from this as in the proof of Proposition~\ref{prop:mainvar}. 
When $s=s_0$, the proof is identical to the proof of Proposition~\ref{prop:main} in the case when $|\xi-\xi_1|, |\xi-\xi_3| \geq 1$ for $B(u)$, $R(u)$ and $NR_2(u)$, and in the case $|\xi-\xi_1|, |\xi_2-\xi_1| \geq 1$ for $NR_1(u)$.

We start with $R(u)$. Denoting $|u_\xi | \la \xi\ra^s$ by $f(\xi)$, and using symmetry, it suffices to prove that
$$
\Big\|\int\limits_{|\xi-\xi_1|\leq 1,\,\, |\xi-\xi_1|\leq |\xi-\xi_3| \atop{|\xi-\xi_1| |\xi-\xi_3|\ll (|\xi-\xi_1|+|\xi-\xi_3|+|\xi|)^{2-2\alpha}}}  \frac{ \la \xi\ra^{s+a}f(\xi_1) f(\xi_1+\xi_3-\xi)f(\xi_3)}{\la \xi_1\ra^s \la \xi_3\ra^s \la \xi_1+\xi_3-\xi\ra^s} d\xi_1d\xi_3\Big\|_{L^2_\xi}\les \|f\|_{L^2}^3.
$$
By Cauchy Schwarz inequality in $\xi_1$ and $\xi_3$ it is enough to consider the case $|\xi|\gg 1$. 
Note that after the change of variable $\xi-\xi_1=\rho$, the left hand side can be bounded by 
$$
\Big\|\int\limits_{|\rho|\leq 1,\,\, |\rho|\leq |\xi-\xi_3| \atop{|\rho| |\xi-\xi_3|\ll ( |\xi-\xi_3|+ | \xi|)^{2-2\alpha}}}  \frac{ |\xi|^{ a}f(\xi-\rho) f(\xi_3-\rho)f(\xi_3)}{  \la \xi_3\ra^{2s}  } d\rho d\xi_3\Big\|_{L^2_{|\xi|\gg 1}}.
$$
We consider two cases.
\\
Case 1. $|\xi_3|\ll |\xi|$. In this region we bound the above integral by
\begin{multline*}
\Big\|\int\limits_{ |\rho|  \ll  |\xi|^{1-2\alpha }}  \frac{ |\xi|^{ a}f(\xi-\rho) f(\xi_3-\rho)f(\xi_3)}{  \la \xi_3\ra^{2s}  } d\rho d\xi_3\Big\|_{L^2_{|\xi|\gg 1}} \\
 \les \Big\|\int\limits_{ |\rho|  \leq 1}  \frac{ f(\xi-\rho) f(\xi_3-\rho)f(\xi_3)}{  |\rho|^{\frac{a}{2\alpha-1}}  } d\rho d\xi_3\Big\|_{L^2_{|\xi|\gg 1}} 
\les \|f\|_{L^2}^3 \int\limits_{ |\rho|  \leq 1}  \frac{ 1}{  |\rho|^{\frac{a}{2\alpha-1}}  } d\rho \les \|f\|_{L^2}^3,
\end{multline*}
provided that $a<2\alpha-1$. In the second inequality we first took the $L^2$ norm inside and then integrated in $\xi_3$. 

\noindent
Case 2. $|\xi_3|\gtrsim |\xi|$. For $a\leq 2s$ we bound the integral by 
\be\label{RCASE2}
\Big\|\int\limits_{ |\rho|  \leq 1}    f(\xi-\rho) f(\xi_3-\rho)f(\xi_3) d\rho d\xi_3\Big\|_{L^2_{|\xi|\gg 1}}   \les \|f\|_{L^2}^3
\ee
where we took the $L^2$ norm inside and integrated as before.

We now consider $B(u)$. Similarly, it suffices to prove
$$
\Big\| \int\limits_{ |\xi-\xi_1|\leq 1,\,\, |\xi-\xi_1|\leq |\xi-\xi_3|  \atop{|\xi-\xi_1| |\xi-\xi_3|\ges ( |\xi-\xi_3|+|\xi|)^{2-2\alpha}}}  \frac{ |\xi|^{ a}  \la \xi_3\ra^{-2s}  f(\xi_1) f(\xi_1+\xi_3-\xi)f(\xi_3)}{| \xi^{2\alpha}-\xi_1^{2\alpha} +(\xi_1+\xi_3-\xi)^{2\alpha}- \xi_3^{2\alpha} | } d\xi_1d\xi_3 \Big\|_{L^2_{|\xi|\gg 1}}   \les \|f\|_{L^2}^3. 
$$
Using Lemma~\ref{freq_est} and letting $\rho=\xi-\xi_1$, we bound the left hand side  by 
$$
\Big\| \int\limits_{ |\rho|\leq 1,\,\, |\rho|\leq |\xi-\xi_3|  \atop{|\rho| |\xi-\xi_3|\ges ( |\xi-\xi_3|+|\xi|)^{2-2\alpha}}}  \frac{ |\xi|^{ a}    ( |\xi-\xi_3|+|\xi|)^{2-2\alpha}  f(\xi-\rho) f(\xi_3-\rho)f(\xi_3)}{\la \xi_3\ra^{ 2s} |\rho| |\xi-\xi_3|  } d\rho d\xi_3 \Big\|_{L^2_{|\xi|\gg 1}}.
$$
We again consider two cases.

\noindent
Case 1. $|\xi_3|\ll |\xi|$. In this region we bound the integral above by
\begin{multline*}
\Big\| \int\limits_{   1\geq   |\rho|  \ges |\xi|^{1-2\alpha }}  \frac{ |\xi|^{ a+1-2\alpha}  f(\xi-\rho) f(\xi_3-\rho)f(\xi_3)}{  |\rho|  } d\rho d\xi_3 \Big\|_{L^2_{|\xi|\gg 1}}
\\ \les \Big\| \int\limits_{   1\geq   |\rho| }  \frac{ |\xi|^{ a+1-2\alpha+}  f(\xi-\rho) f(\xi_3-\rho)f(\xi_3)}{  |\rho|^{1-}  } d\rho d\xi_3 \Big\|_{L^2_{|\xi|\gg 1}} \les \|f\|_{L^2}^3,
\end{multline*}
provided that $a<2\alpha-1$. In this case the last inequality follows by integrating in $\xi, \xi_3, \rho$, in that order.

\noindent
Case 2. $|\xi_3|\gtrsim |\xi|$. Using $|\rho| |\xi-\xi_3|\ges ( |\xi-\xi_3|+|\xi|)^{2-2\alpha}$, we bound the integral by \eqref{RCASE2} for $a\leq 2s$.

For $NR_2(u)$, note that by the algebra property of Sobolev spaces
$$
\widetilde f(\xi-\xi_1-\xi_3)= \Big|\int \la \xi-\xi_1-\xi_3\ra^s \overline{u_{\xi_2}} \overline{u_{\xi_4}} u_{\xi-\xi_1+\xi_2-\xi_3+\xi_4} d\xi_2 d\xi_4 \Big| \in L^2_{\xi-\xi_1-\xi_3}.
$$
This reduces the proof for $NR_2(u)$ to the case of $B(u)$. 

Finally for $NR_1(u)$, the analysis is similar by  eliminating $\xi_5$ and then integrating in $\xi_3$ and $\xi_4$. Then, the change of  variable $\xi_3=\xi-\xi_1+\xi_2$ reduces its proof to the case of $B(u)$. 
\end{proof}

   \section*{Appendix. Quintic NLS on the Real Line}
We consider the quintic NLS equation on the real line:  
 \begin{equation}\label{qsch}
\left\{
\begin{array}{l}
iu_{t}+\Delta u =\pm |u|^4u, \,\,\,\,  x \in {\mathbb{R}}, \,\,\,\,  t\in \mathbb{R} ,\\
u(x,0)=u_0(x)\in H^{s}(\mathbb{R}). \\
\end{array}
\right.
\end{equation}
We prove the following local estimate: For any $s>\frac13$ and  $a< \min(3s-1,\frac12)$, we have 
\be\label{quin_bound} \big\||u|^4u\big\|_{X^{s+a,-\frac12+}}\les \|u\|^5_{X^{s,\frac12+}}.
\ee
This leads to a smoothing estimate as in Theorem~\ref{thm:1}. 
The equation is $L^2-$critical and one can control the $H^1$ norm  in the defocusing case. Dodson in \cite{bd} obtained global solutions for $s>\frac14$  and provided polynomial-in-time bounds for any $\frac14<s\leq 1$. Therefore, in the defocusing case our smoothing theorem can be extended to all times.   

   As in Section 3, to obtain \eqref{quin_bound} the bound, it suffices to bound the multiplier 
   $$
  \sup_{\xi} \int\limits_{\xi-\xi_1+\xi_2-\xi_3+\xi_4-\xi_5=0} \frac{\la\xi\ra^{2s+2a} \prod_{i=1}^5 \la \xi_i\ra^{-2s} }{\la \xi^2-\xi_1^2+\xi_2^2-\xi_3^2+\xi_4^2-\xi_5^2\ra^{1-}}.
   $$
 Without loss of generality $|\xi_5| \gtrsim |\xi|$; the proof is similar in the other cases.   In this case, by substituting  $\xi_5=\xi-\xi_1+\xi_2-\xi_3+\xi_4$, we bound the integral by 
  $$
   \int \frac{\la\xi\ra^{ 2a} \prod_{i=1}^4 \la \xi_i\ra^{-2s} }{\la 2(\xi-\xi_1+\xi_2-\xi_3)\xi_4 +p(\xi,\xi_1,\xi_2,\xi_3)  \ra^{1-}}   d\xi_1 d\xi_2 d\xi_3 d\xi_4,
  $$ 
	where $p$ is a  polynomial of degree $2$. 
Note that (for $s>0$)
$$
\int \frac{\la x\ra^{-2s}}{\la Ax+B\ra^{1-}} dx \les \frac1{|A|^{1-}} \int \frac{\la x\ra^{-2s}}{ |x+B/A|^{1-} } dx \les \frac1{|A|^{1-}},
$$	
uniformly in $B$. 
Using this in the $\xi_4$ integral,  we bound the integral above by 
$$
   \int \frac{\la\xi\ra^{ 2a} \prod_{i=1}^3 \la \xi_i\ra^{-2s} }{|\xi-\xi_1+\xi_2-\xi_3|^{1-}}   d\xi_1 d\xi_2 d\xi_3 \les  \int \frac{\la\xi\ra^{ 2a} \prod_{i=1}^2 \la \xi_i\ra^{-2s} }{\la \xi-\xi_1+\xi_2\ra^{\min(2s,1)-}}   d\xi_1 d\xi_2 \les \frac{ \la \xi\ra^{2a}}{\la \xi\ra^{\min(6s-2,1)-}}.
 $$ 
The last term is bounded for any  $a<\min(3s-1,\frac12)$.

 \end{document}